\documentclass[a4paper,11pt]{amsart}

\usepackage[english]{babel}
\usepackage{dsfont} 

\usepackage{amssymb} 
\usepackage{amsthm} 

\newtheorem{thm}{Theorem}
\newtheorem{dfn}[thm]{Definition}
\newtheorem{lem}[thm]{Lemma}

\newtheorem{prop}[thm]{Proposition}

\newcommand{\C}{\mathds{C}}
\newcommand{\R}{\mathds{R}}
\newcommand{\N}{\mathds{N}}

\newcommand{\cB}{\mathcal{B}}

\newcommand{\cJ}{\mathcal{J}}

\newcommand{\cN}{\mathcal{N}}
\newcommand{\cM}{\mathcal{M}}

\newcommand{\cK}{\mathcal{K}}

\newcommand{\cT}{\mathcal{T}}
\newcommand{\cI}{\mathcal{I}}

\newcommand{\cZ}{\mathcal{Z}}

\newcommand{\ux}{\underline{x}}

\newcommand{\ov}{\overline}
\newcommand{\ii}{\sf{i}}

\newcommand{\Z}{\mathds{Z}}

\newcommand{\cA}{\mathcal{A}}

\newcommand{\cL}{\mathcal{L}}

\newcommand{\Lin}{\rm{Lin}}

\author{Konrad Schm\"udgen}
\address{Universit\"at Leipzig, Mathematisches Institut, Augustusplatz 10/11, D-04109 Leipzig, Germany}
\email{schmuedgen@math.uni-leipzig.de}

\begin{title}
{A characterization of moment sequences by extensions}
\end{title}
\date{}
\begin{document}

\maketitle

\begin{abstract}
The fibre theorem \cite{schm2003} for the moment problem on closed semi-algebraic subsets of $\R^d$ is generalized to finitely generated real unital algebras. A charaterization of moment functionals on the polynomial algebra $\R[x_1,\dots,x_d]$ in terms of extensions is given. These results are applied to reprove basic results on the complex moment problem due to  Stochel and Szafraniec \cite{stochelsz} and Bisgaard \cite{bisgaard}.

\end{abstract}


\textbf{Key  words:} moment problem, real algebraic geometry

\section {Introduction}

The present paper deals with the classical moment problem on  $\R^d$. A useful result on the existence of solutions is the fibre theorem \cite{schm2003} for  closed semi-algebraic subsets of $\R^d$. The proof given  in \cite{schm2003} was based on the decomposition theory of states on $*$-algebras. An elementary proof has been found   by T. Netzer \cite{netzer}, see also    M. Marshall \cite[Chapter 4]{marshall}. 

In this paper we first generalize the fibre theorem from the polynomial algebra $\R[x_1,\dots,x_d]$ to arbitary finitely generated unital real algebras (Theorem \ref{fibreth1}). This general version is the starting point for a number of  applications. The main result of the paper is an  extension theorem (Theorem \ref{extensionthm}) that provides a necessary and sufficient condition for a linear functional on $\R[x_1,\dots,x_d]$ being a moment functional. In the case $d=2$ this result leads to a basic theorem of J. Stochel and F. H. Szafraniec \cite{stochelsz} on the complex moment problem (Theorem  \ref{extcmpstochelsz}). Another application of the generalized fibre theorem is a short proof of a theorem of T.M. Bisgaard on the two-sided complex moment problem (Theorem \ref{bisgaardthm}).

Let us fix some definitions and notations which are used throughout this paper.
A {\it complex $*$-algebra} $\cB$ is a complex algebra equipped with 
an involution, that is, 
an antilinear mapping $\cB\ni b\to b^*\in \cB$ satisfying $(bc)^*= c^*b^*$ and $(b^*)^*=b$  for\, $b,c\in \cB$. 
Let $\sum \cB^2$ denote the set of finite sums $\sum_j b_j^*b_j$ of hermitean squares $b_j^*b_j$, where $b_j\in \cB$. A linear functional $L$ on $\cB$ is called {\it positive} if $L$ is nonnegative on $\sum \cB^2$, that is, if\, $L(b^*b)\geq 0$\, for all $b\in \cB$.

A  {\it $\ast$-semigroup}\, is a semigroup $S$ with a mapping\, $s^*\to s$\, of\, $S$\, into itself, called  involution such that $(st)^*=t^*s^*$ and $(s^*)^*=s$ for $s,t\in S$. The {\it semigroup $\ast$-algebra} $\C S$ of $S$ is the complex $*$-algebra is the vector space  of all finite sums $\sum_{s\in S} \alpha_s s$, where $\alpha_s\in \C$, with product and involution 
\begin{align*}
\big(\sum\nolimits_s\alpha_s s\big)\big(\sum\nolimits_{t} \beta_t t\big):= \sum\nolimits_{s,t} \alpha_s\beta_t st, \quad \big(\sum\nolimits_{s} \alpha_s s\big)^*:=\sum\nolimits_{s} \ov{\alpha}_s\, s^*.
\end{align*}
The polynomial algebra $\R[x_1,\dots,x_d]$ is abbreviated by $\R_d[\ux]$.

\section{A generalization of the fibre theorem}\label{propertiessmpandmpandfibretheorem}

Throughout this section  $\cA$ is a {\it finitely generated commutative real unital algebra}.  By a character  of $\cA$ we mean an algebra homomorphism $\chi:\cA\to \R$ satisfying $\chi(1)=1$. We equipp the set $\hat{\cA}$ of  characters of $\cA$ with the weak topology. 

Let us fix a set $\{f_1,\dots,f_d\}$  of  generators of the algebra $\cA$. Then there is a algebra homomorphism $\pi:\R[\ux]\to \cA$ such that $\pi(x_j)=f_j$,  $j=1,\dots,d.$ If\, $\cJ$ denotes the kernel of $\pi$, then  $\cA$ is isomorphic to the quotient algebra\, $\R_d[\ux]/ \cJ$.
Further, each character $\chi$ is completely determined by  the point $x_\chi:=(\chi(f_1),\dots,\chi(f_d))$ of $\R^d$. For simplicity we will identifiy $\chi$ with $x_\chi$ and write $f(x_\chi):=\chi(f)$ for $f\in \cA$. Then\, $\hat{\cA}$ becomes a real algebraic subvariety of $\R^d$. In the special case $\cA=\R[x_1,\dots,x_d]$ we can take $f_1=x_1,\dots,f_d=x_d$\, and obtain $\hat{\cA}\cong \R^d$. 
\begin{dfn}
A {\rm preorder} of\, $\cA$ is a subset $\cT$ of $\cA$ such that 
\begin{align*}
\cT\cdot \cT\subseteq \cT,~~\cT + \cT \subseteq \cT,~~ 1 \in \cT,~~
a^2 \cT  \in \cT~~{\rm for~all}~~a\in \cA.
\end{align*} 
\end{dfn}
Let $\sum \cA^2$ denote  the set of  finite sums $\sum_i a_i^2$ of squares of elements $a_i\in \cA$.
 Since $\cA$ is commutative,   $\sum \cA^2$ is invariant under multiplication and hence  the smallest  preorder of $\cA$. 

For a preorder  $\cT$  of $\cA$, we define 
\begin{align*}
\cK(\cT)=\{ x\in \hat{\cA}:  f(x)\geq 0 ~~ {\rm for~all}~f\in \cT\}.
\end{align*}
The main concepts are introduced in the following definition, see \cite{schm2003} or \cite{marshall}. 
\begin{dfn}\label{definionsmpmp}
A preorder $\cT$ of\, $\cA$ has the\\
$\bullet$  \emph{moment property (MP)}\, if   each $\cT$-positive linear functional $L$ on $\cA$ is a moment functional, that is, there exists a positive Borel measure $\mu\in \cM(\hat{\cA})$ such that
\begin{align}\label{Lrepreborelmu}
L(f) =\int_{\hat{\cA}}\, f(x) \, d\mu(x) \quad {\rm for~all}~~f\in \cA,
\end{align} 
$\bullet$  \emph{strong moment property (SMP)} if  each $\cT$-positive linear functional $L$ on $\cA$ is a $\cK(\cT)$--moment functional, that is, there is a positive Borel measure $\mu\in \cM(\hat{\cA})$ such that\, 
${\rm supp}\, \mu\subseteq \cK(\cT)$\, and
(\ref{Lrepreborelmu}) holds.
\end{dfn}

To state our main result (Theorem \ref{fibreth1})  we need some preparations.

Suppose that $\cT$ a finitely generated preorder of $\cA$ and
let ${\sf f}=\{f_1,\dots,f_k\}$ be a  sequence of generators of  $\cT$. 
We consider a fixed $m$-tuple\, ${\sf h}=(h_1,\dots,h_m)$\,  of  elements $h_k\in \cA$. Let  $\ov{{\sf h}(\cK(\cT))}$ denote the closure of the subset ${\sf h}(\cK(\cT))\subseteq \hat{\cA}$  in $\hat{\cA}$, where ${\sf h}(\cK(\cT))$ is defined by
$$
 {\sf h}(\cK(\cT))=\{(h_1(x),\dots,h_m(x)); x \in \cK(\cT)\}.
 $$
For $\lambda=(\lambda_1,\dots,\lambda_r)\in \R^m $ we denote by $\cK(\cT)_\lambda$ the semi-algebraic set
$$
\cK(\cT)_\lambda=\{ x\in \cK(\cT): h_1(x)=\lambda_1,\dots,h_m(x)=\lambda_m \}
$$ 
and by\, $\cT_\lambda$ the  preorder generated by the sequence
$${\sf f}(\lambda):= \{f_1,\dots,f_k,h_1-\lambda_1,\lambda_1-h_1,\dots,h_m-\lambda_m,\lambda_m
-h_m\}$$
Then $\cK(\cT)$ is  the disjoint union of fibre set $\cK(\cT)_\lambda=\cK(\cT_\lambda)$, where $\lambda \in{\sf h}(\cK(\cT))$.

Let $\cI_\lambda$ be the ideal of $\cA$ generated by $h_1-\lambda_1,\dots,h_m-\lambda_m$. Then we have
$\cT_\lambda:=\cT +\cI_\lambda$ and  the preorder $\cT_\lambda/\cI_\lambda$ of the quotient algebra $\cA/\cI_\lambda$  is generated by 
$$
\pi_\lambda({\sf f}):=\{\pi_\lambda(f_1),\dots,\pi_\lambda(f_k)\},$$
where $\pi_\lambda:\cA\to \cA/\cI_\lambda$ denotes the canonical map. 

Further, let $\hat{\cI}_\lambda:=\cI(\cZ(\cI_\lambda))$ denote the ideal of all elements $f\in \cA$ which vanish on the zero set $\cZ(\cI_\lambda)$ of $\cI_\lambda$. Clearly, $\cI_\lambda \subseteq \hat{\cI}_\lambda$ and $\cZ(\cI_\lambda)=\cZ(\hat{\cI}_\lambda)$.
Set $\hat{\cT}_\lambda:=\cT +\hat{\cI}_\lambda$. Then $\hat{\cT}_\lambda/\hat{\cI}_\lambda$ is a preorder of the quotient algebra $\cA/\hat{\cI}_\lambda$.

Note that in  general we have $\cI_\lambda \neq \hat{\cI}_\lambda$ and equality holds if and only if the ideal $\cI_\lambda$ is {\it real}. The latter means that $\sum_j a_j^2\in \cI_\lambda$ for finitely  many elements $a_j\in \cA$ always implies that $a_j\in \cI_\lambda$ for all $j$.
\begin{thm}\label{fibreth1} 
Let $\cA$ be a finitely generated commutative real unital algebra and let $\cT$ be a finitely generated preorder of  $\cA$. Suppose that  $h_1,\dots,h_m$ are elements of $\cA$ that are bounded on the  set  $\cK(\cT)$. Then the following are equivalent:\\
(i) $\cT$ has property (SMP) (resp. (MP)) in $\cA$.\\
(ii)  $\cT_\lambda$ satisfies (SMP) (resp. (MP)) in\, $\cA$\, for all\, $\lambda \in {\sf h}(\cK(\cT)).$\\
$(ii)^\prime$  $\hat{\cT}_\lambda$ satisfies (SMP) (resp. (MP)) in\, $\cA$\, for all\, $\lambda \in {\sf h}(\cK(\cT)).$\\
(iii) $\cT_\lambda/\cI_\lambda$ has (SMP) (resp. (MP)) in\, $\cA/\cI_\lambda$\, for all\, $\lambda \in {\sf h}(\cK(\cT)).$\\
$(iii)^\prime$   $\hat{\cT}_\lambda/\hat{\cI}_\lambda$ has (SMP) (resp. (MP)) in\, $\cA/\hat{\cI}_\lambda$\, for all\, $\lambda \in {\sf h}(\cK(\cT)).$ 
\end{thm}

 Proposition \ref{quotientsmpmp}(i) gives
the  implication (i)$\to$(ii),  Proposition \ref{quotientsmpmp}(ii) yields   the equivalences (ii)$\leftrightarrow$(iii) and (ii)$^\prime\leftrightarrow$(iii)$^\prime$, and  Proposition \ref{quotientsmpmp}(iii) implies equivalence (ii)$\leftrightarrow$(ii)$^\prime$.

The  main assertion of Theorem \ref{fibreth1} is the implication (ii)$\to$(i). Its proof   is lenghty and technically involved. 
In the proof given below we reduce the general case to the case $\R_d[\ux]$.

\begin{prop}\label{quotientsmpmp}
Let $\cI$ be an ideal and $\cT$  a finitely generated preorder  of $\cA$. Let $\cI$ be the ideal of all $f\in \cA$ which vanish on the zero set $\cZ(\cI)$ of $\cI$.\\
(i) If\,  $\cT$ satisfies (SMP) (resp. (MP)) in $\cA$,  so does   $\cT+\cI$.\\
(ii) $\cT+\cI$ satisfies (SMP) (resp. (MP)) in $\cA$ if and only if 
$(\cT+\cI)/\cI$ does in $\cA/\cI$.\\
(iii)  $\cT+\cI$ obeys (SMP) (resp. (MP)) in $\cA$ if and only if $\cT+\hat{\cI}$ does.
\end{prop}
\begin{proof}
(i): Seee e.g. \cite{scheiderer}, Proposition 4.6. 

(ii) See e.g. \cite{scheiderer},  Lemma 4.5.

(iii): It suffices to show that both preorders $\cT+\cI$ and $\cT+\hat{\cI}$ have the same positive characters and the same positive linear functionals on $\cA$. For the sets of characters, using the equality $\cZ(\cI)=\cZ(\hat{\cI})$ we obtain 
\begin{align*}
\cK(\cT+\cI)=\cK(\cT)\cap \cZ(\cI)=\cK(\cT)\cap \cZ(\hat{\cI})=\cK(\cT+\hat{\cI}).
\end{align*}
Since $\cT+\cI\subseteq \cT+\hat{\cI}$, a $(\cT+\hat{\cI})$-positive functional is trivially $(\cT+\cI)$-positive.
 
Conversely, let $L$ be a\, $(\cT+\cI)$-positive linear functional on $\cA$ and let $a\in \hat{\cI}$. Set  $a^\prime:=\pi^{-1}(a)$. Recall that  $\cA\cong \R[\ux]/ \cJ$ as noted  above. Clearly, $\cI^\prime:=\pi^{-1}(\cI)$ is an ideal of $\R_d[\ux]$ and we have $\cZ(\cI^\prime)=\cZ(\cJ)\cap \cZ(\cI)$. Hence, since $\pi(a)=a^\prime$, 
the polynomial $a^\prime\in \R_d[\ux]$ vanishes on $\cZ(\cI^\prime)$. Therefore, by the real Nullstellensatz \cite[Theorem 2.2.1, (3)]{marshall}, there are $m\in \N$ and $b\in \sum \R_d[\ux]^2$ such that $c^\prime:=(a^\prime)^{2m}+b\in \cI^\prime$. Upon multiplying $c^\prime$ by some  even power of $a^\prime$ we can assume that $2m=2^k$ for some $k\in\N$. Then   
\begin{align}\label{cprimepib}
c:=\pi(c^\prime)=a^{2^k}+\pi(b)\in \cI ,\quad {\rm  where}\quad \pi(b)\in \sum \cA^2.
\end{align} 
Being $(\cT+\cI)$-positive, $L$  annihilates $\cI$ and  is nonnegative on $\sum \cA^2$. Therefore, by (\ref{cprimepib}),
$$
0=L(c)=L\big(a^{2^k}\big)+ L(\pi(b)),~~ L(\pi(b))\geq 0,~~ L\big(a^{2^k}\big)\geq 0.
$$
Hence $L(a^{2^k})=0$.
Since $L$ is nonnegative on $\sum \cA^2$, the Cauchy-Schwarz inequality holds. By a repeated application of this inequality we derive  
\begin{align*}
|L(a)|^{2^k}&\leq L(a^2)^{2^{k-1}} L(1)^{2^{k-1}}\leq L(a^4)^{2^{k-2}} L(1)^{2^{k-2}+2^{k-1}}\leq \dots \\&\leq L(a^{2^k}) L(1)^{1+\dots+2^{k-1}} =0.
\end{align*}
Thus $L(a)=0$.  That is, $L$ annihilates $\hat{\cI}$. Hence $L$ is $(\cT+\hat{\cI})$-positive which  completes the proof of (iii): 
\end{proof}

\noindent
{\it Proof of the implication (ii)$\to$(i) of Theorem \ref{fibreth1}:}\\
As noted at the beginning of this section 
 there is an algebra homomorphism $\pi:\R[\ux]\to \cA$ such that $\pi(x_j)=f_j$,  $j=1,\dots,d,$ 
and $\cA$ is isomorphic to the quotient algebra\, $\R_d[\ux]/ \cJ$, where $\cJ$ is the kernel of $\pi$.
 We  choose polynomials $\tilde{h}_j\in \R_d[\ux]$ such that $\pi(\tilde{h}_j)=h_j$. Clearly, $\tilde{\cT}:=\pi^{-1}(\cT)$ is a  preorder of $\R_d[\ux]$ such that $\tilde{\cT}=\cT+\cJ$ and $\cK(\tilde{\cT})=\cK(\cT)$. Hence each polynomial\, $\tilde{h}_j$\, is bounded on $\cK(\tilde{\cT})$ and\, $ {\sf \tilde{ h}}(\cK(\tilde{\cT}))= {\sf h}(\cK(\cT)).$\, Further, $\tilde{\cT}_\lambda:=\tilde{\cT}+\cI_\lambda= \cT_\lambda+\cJ$\, is also a preorder of $\R_d[\ux]$. By Proposition \ref{quotientsmpmp}(ii), $\tilde{\cT}_\lambda$ resp. $\tilde{\cT}$ has (SMP) (resp. (MP)) in\, $\R_d[\ux]$\, if and only $\cT_\lambda$ resp. $\cT$ does in $\cA=\R_d[\ux]/\cJ$. Therefore, the assertion (ii)$\to$(i)  of Theorem \ref{fibreth1} follows from the corresponding result for the preorders $\tilde{\cT}$ and $\tilde{\cT}_\lambda$ of the algebra $\R_d[\ux]$ proved in \cite{schm2003}. \hfill $\Box$

\smallskip

Theorem \ref{fibreth1} is  formulated for
a commutative unital {\it real} algebra $\cA$. In Sections \ref{applcmp} and \ref{twosdidecmp} we are concerned with commutative {\it complex} semigroup $*$-algebras.
This case can be easily reduced to Theorem \ref{fibreth1} as we discuss in what follows. 

If $\cB$ is a  commutative complex $*$-algebra, its hermitean part $$\cA=\cB_h:=\{ b\in \cB:b=b^*\}$$ is a commutative real algebra.

Conversely, suppose that $\cA$ is a commutative real algebra. Then  
its complexification $\cB:=\cA+ i\cA$  is a commutative  complex $*$-algebra with involution $(a_1+i a_2)^*:=a_1-i a_2$ and  scalar multiplication $$(\alpha+i\beta)(a_1+i a_2):= \alpha a_1-\beta a_2 + i (\alpha a_2 +\beta a_1),~~  \alpha,\beta \in \R,~a_1,a_2\in \cA,$$  and  $\cA$ is the hermitean part $\cB_h$ of $\cB$.  Let $b\in \cB$. Then we have $b=a_1+i a_2$ with $a_1,a_2\in \cA$ and since $\cA$  is commutative, we get
\begin{align}\label{sumba^2}
 b^*b=(a_1-i a_2)(a_1+i a_2)=a_1^2+a_2^2 +i(a_1a_2-a_2a_1) =a_1^2+a_2^2.
\end{align}
 Hence, if $\cT$ is a preorder of $\cA$, then\, $b^*b\,\cT\subseteq \cT$\, for $b\in \cB$. In particular,  $\sum \cB^2=\sum \cA^2$.

Further, each $\R$-linear functional $L$ on $\cA$ has a unique extension $\tilde{L}$ to a $\C$-linear functional  on $\cB$. By (\ref{sumba^2}), $L$ is nonnegative on $\sum\cA^2$ if and only of  $\tilde{L}$ is nonnegative on $\sum \cB^2$, that is, $\tilde{L}$ is a positive linear functional on $\cB$.
\section{An extension theorem}\label{extensionsection}
In this section we derive a theorem which characterizes  moment functional on $\R^d$ in terms of extensions.

Throughout let $\cA$ denote the real algebra of functions on $(\R^d)^\times:=\R^d\backslash \{0\}$ generated by the polynomial algebra $\R_d[\ux]$ and the functions
\begin{align}\label{sumfkl1}
f_{kl}(x):=x_kx_l(x_1^2+\dots+x_d^2)^{-1},~~{\rm  where}~~ k,l=1,\dots,d,\, x\in\R^d\backslash \{0\} .
\end{align}
Clearly, these functions satisfy the identity
\begin{align}
\sum_{k,l=1}^d\, f_{kl}(x)^2=1.
\end{align}
That is, the functions $f_{kl}$, $k,l=1,\dots,d$, generate the coordinate algebra $C(S^d)$ of the unit sphere $S^d$ in $\R^d$.
The next lemma describes  the character set $\hat{\cA}$  of $\cA$.
\begin{lem}\label{charactersetahat}
The set $\hat{\cA}$ is parametrized by the disjoint union of $\R^d\backslash \{0\} $ and $ S^d$. For\, $x\in \R^d\backslash \{0\}$\, the  character $\chi_x$  is the  evaluation of functions  at $x$ and  for $t\in S^d$ the character $\chi^t$ acts by $\chi^t(x_j)=0$ and $\chi^t(f_{kl})=f_{kl}(t)$, where $j,k,l=1,\dots,d$.
 \end{lem}
 \begin{proof}
It is obvious that for any $x\in \R^d\backslash \{0\}$ the point evaluation $\chi_x$ at $x$ is a character on the algebra $\cA$ satisfying $(\chi(x_1),\dots,\chi(x_d))\neq 0$.  

Conversely, let $\chi$  be a character of $\cA$ such that
 $x:=(\chi(x_1),\dots,\chi(x_d))\neq 0$.
Then the identity $(x_1^2+\dots+x_d^2)f_{kl}=x_kx_l$ implies  that
$$(\chi(x_1)^2+\dots+\chi(x_d)^2)\chi(f_{kl})=\chi(x_k)\chi(x_l)$$ and therefore
$$\chi(f_{kl})=(\chi(x_1)^2+\dots+\chi(x_d)^2)^{-1}\chi(x_k)\chi(x_l)=f_{kl}(x).$$
Thus  $\chi$ acts on the generators $x_j$ and $f_{kl}$, hence on the whole algebra $\cA$, by point evaluation at $x$, that is, we have $\chi=\chi_x$.

Next let us note that the quotient of $\cA$ by the ideal generated by $\R_d[\ux]$ is (isomorphic to) the algebra $C(S^d)$.
Therefore, if   $\chi$ is a chacacter of $\cA$ such that $(\chi(x_1),\dots,\chi(x_d))=0$, then it gives  a character on the algebra $C(S^d)$. Clearly, each character of $C(S^d)$ comes from a point  of $S^d$. Conversely, each point  $t\in S^d$ defines a unique character of $\cA$ by $\chi^t(f_{kl})=f_{kl}(t)$ and $\chi^t(x_j)=0$ for all $k,l,j$.
\end{proof}

 \begin{thm}\label{mpshperesinrd}
 The preorder $\sum \cA^2$ of the algebra $\cA$ satisfies (MP), that is, for each positive linear functional $\cL$ on $\cA$ there exist positive Borel measures $\nu_0$ on $S^d$ and $\nu_1\in \cM(\R^d\backslash \{0\}) $ such that for all polynomials $g$ we have
 \begin{align}\label{descptionltilde}
 \cL&(g(x,f_{11}(x),\dots, f_{dd}(x)))\nonumber\\&=\int_{S^d} g(0,f_{11}(t),\dots,f_{dd}(t)) \, d\nu_0(t) +\int_{\R^d\backslash \{0\}} g(x,f_{11}(x),\dots,f_{dd}(x)) ~ d\nu_1(x).
 \end{align}
 \end{thm}
 \begin{proof}
 It suffices to prove that $\sum \cA^2$ has (MP). The  assertions follow  then from the definition of the property (MP) and the explicit form  of the character set  given in Lemma \ref{charactersetahat}.
 
 From the description of  $\hat{\cA}$ it is obvious that the functions $f_{kl}$, $k,l=1,\dots,d,$ are bounded on $\hat{\cA}$, so we can take them as polynomials $h_j$ in Theorem \ref{fibreth1}. Consider a  non-empty fibers  for  $\lambda=(\lambda_{kl})$, where $\lambda_{kl}\in \R$, and let $\chi\in \hat{\cA}$ be such that $\chi(f_{kl})=\lambda_{kl}$ for all $k,l$. If $\chi=\chi^t$, then $\chi(x_j)=0$ for all $j$ and hence $\cA/\cI_\lambda=\C\cdot 1$ has  trivially (MP). Now suppose that $\chi=\chi_x$ for some $x\in\R^d\backslash \{0\}$. Then $\chi_x(f_{kl})=f_{kl}(x)=\lambda_{kl}$. Since  $1=\sum_{k} f_{kk}(x)=\sum_k \lambda_{kk}$, there is a $k$ such that $\lambda_{kk}\neq 0$. From the equality   $\lambda_{kk}=f_{kk}(x)=x_k^2(x_1^2+\dots+x_d^2)^{-1}$ it follows that $x_k\neq 0$. Hence  $\frac{\lambda_{kl}}{\lambda_{kk}}=\frac{f_{kl}(x)}{f_{kk}(x)}=\frac{x_l}{x_k}$, so that $x_l=\frac{\lambda_{kl}}{\lambda_{kk}}\,x_k$ for all $l=1\dots,d$. This implies that  the quotient algebra\, $\cA/\cI_\lambda$\, of $\cA$ by the fiber ideal\, $\cI_\lambda$\,  is an algebra of polynomials in the songle variable $x_k$. Therefore, the preorder $\sum (\cA/\cI_\lambda)^2$ satisfies (MP).  Hence $\cA$ itself obeys (MP) by  Theorem \ref{fibreth1}.
\end{proof}

The main result of this  section is the following extension theorem.
\begin{thm}\label{extensionthm}
 A linear functional $L$ on $\R_d[\ux]$ is a  moment functional if and only if it has an extension to a positive linear functional $\cL$ on the larger algebra $\cA$. 
\end{thm}
\begin{proof} 
Assume first that $L$ has an extension to a positive linear functional $\cL$ on $\cA$. By Theorem \ref{mpshperesinrd}, the functional $\cL$ on $\cA$ is of the form described by equation (\ref{descptionltilde}). We define a positive Borel measure $\mu$ on $\R^d$ by $$\mu(\{0\})=\nu_0(S^d),\quad  \mu(M\backslash \{0\})=\nu_1(M\backslash \{0\}).$$ Let $p\in \R_d[\ux]$. Setting $g(x,0,\dots,0)=p(x)$ in (\ref{descptionltilde}), we get 
$$
L(p)=\cL(p)= \nu_0(\{0\})p(0)+ \int_{\R^d\backslash \{0\}} g(x,0,\dots,0) ~ d\nu_1(x)=\int_{\R^d} p(x)\, d\mu(x).
$$
Thus $L$ is moment functional on $\R_d[\ux]$ with representing measure $\mu$.

Conversely, suppose that  $L$ is a moment functional on $\R_d[\ux]$ and let $\mu$ be a  representing measure. 
Since $f_{kl}(t,0,\dots,0)=\delta_{k1}\delta_{l1}$ for $t\in \R, t\neq 0,$  we have 
$\lim_{t\to 0} f_{kl}(t,0,\dots,0)=\delta_{k1}\delta_{l1}$. Hence there is  a well-defined character on the algebra $\cA$ given by
 $$\chi(f)=\lim_{t\to 0} f(t,0,\dots,0),\quad f\in \cA,$$
  and $\chi(p)=p(0)$ for $p\in \R[\ux]$. Then, for $f\in \R_d[\ux]$, we have
\begin{align}\label{defltilde}
L(f)=\mu(\{0\})\chi(f)+\int_{\R^d\backslash \{0\}} f(x)\, d\mu(x).
\end{align}
For $f\in \cA$ we define $\cL(f)$ by the right-hand side of (\ref{defltilde}). Then $\cL$ is a  positive linear functional on $\cA$ which extends $L$. 
\end{proof}
 Remarks.\\ 1. Another type of extension theorems has been  derived in \cite{pv}. The main difference  to the above theorem is that in \cite{pv}, see e.g. Theorem 2.5, a function $$h(x):=(1+x_1^2+\dots+x_d^2+p_1(x)^2+\dots+p_k(x)^2)^{-1}$$ is added to the algebra, where $p_1,\dots,p_k\in \R_d[\ux]$ are fixed.  Since $h(x)$ is  bounded  on the  character set and  non-empty fibres  consist of  single points, the existence assertions of these results follow immediately from  Theorem \ref{fibreth1}. In this case the representing measure of the extended functional is unique (see \cite[Theorem 2.5]{pv}.

2. The measure $\nu_1$ in Theorem \ref{mpshperesinrd} and hence the representing measure $\mu$ for the functional $\cL$ in Theorem \ref{extensionthm} are  not uniquely determined by $\cL$. (A counter-example is easily constructed by taking some appropriate measure supported by a coordinate axis.)  Let $\mu_{\rm rad}$ denote the measure on $[0,+\infty)$  obtained by transporting $\mu$  by the mapping $x\to \|x\|^2$. Then, as shown in \cite[p. 2964, Nr 2.]{ps}, if  $\mu_{\rm rad}$ is determinate on $[0,+\infty)$, then $\mu$ is is uniquely determined by $\cL$.

\section{Application to the complex moment problem}\label{applcmp}
Given a  complex $2$-sequence $s=(s_{m,n})_{(m,n)\in \N_0^2}$ the  complex moment problem asks when does there exist a positive Borel measure $\mu$ on $\C$ such that 
the function $z^m\ov{z}^n$ on $\C$ is $\mu$-integrable and
\begin{align}\label{cmp}
s_{mn} =\int_\C~ z^m\ov{z}^n\, d\mu(z)\quad {\rm for~all}\quad (m,n)\in \N_0^2 .
\end{align}

The semigroup algebra  $\C \N_0^2$ of the $*$-semigroup $\N_0^2$ with involution $(m,n):=(n,m)$, $(m,n)\in \N_0^2$, is the $*$-algebra $\C[z,\ov{z}]$ with involution given by $z^*=\ov{z}$. If $L$ denotes the linear functional on $\C[z,\ov{z}]$ defined by 
\begin{align*}
L(z^m\ov{z}^n)=s_{m,n},\quad (m,n)\in \N_0^2
\end{align*} then (\ref{cmp}) means that
\begin{align*}
L_s(p)=\int_\C \, p(z,\ov{z})\, d\mu(z), \quad p\in\C[z,\ov{z}].
\end{align*}
Clearly,  $\N_0^2$ is a subsemigroup of the larger $*$-semigroup 
$$
\cN_+=\{(m,n)\in \Z^2: m+n\geq 0\} \quad {\rm with ~involution}\quad (m,n)^*=(n,m).
$$
The following fundamental theorem was proved by J. Stochel and F.H. Szafraniec \cite{stochelsz}.
\begin{thm}\label{extcmpstochelsz}
A linear functional $L$ on $\C[z,\ov{z}]$ is a moment functional if and only if $L$ has an extension to a positive linear functional $\cL$ on the $*$-algebra $\C\cN_+$. 
\end{thm}
In \cite{stochelsz} this theorem was stated 
in terms of semigroups:\smallskip

{\it A complex sequence $s=(s_{m,n})_{(m,n)\in \N_0^2}$   is a moment sequence on $\N_0^2$ if and only if 
there exists a positive semidefinite sequence\,  $\tilde{s}=(\tilde{s}_{m,n})_{(m,n)\in \cN_+}$\,  on the $*$-semigroup $\cN_+$ such that $\tilde{s}_{m,n}=s_{m,n}$ for all $(m,n)\in \N_0^2$.}
\medskip

In order to prove Theorem \ref{extcmpstochelsz} we first describe the semigroup $*$-algebra $\C\cN_+$. Clearly,   $\C\cN_+$ is the complex $*$-algebra   generated by the functions $z^m\ov{z}^n$ on $\C\backslash \{0\}$, where $m,n\in \Z$ and $m+n\geq 0$.
 If $r(z)$ denotes the modulus and $u(z)$ the phase of $z$, then  $z^m\ov{z}^n=r(z)^{m+n}u(z)^{m-n}$. Setting $k=m+n$, it follows that
 $$\C\cN_+=\Lin \{ r(z)^k u(z)^{2m-k};k\in \N_0,m\in \Z\}\,.
 $$ 
 The functions $r(z)$ und $u(z)$ itself are not in $\C\cN_+$, but $r(z)u(z)=z$ and $v(z):=u(z)^2=z \ov{z}^{-1}$ are in $\C\cN_+$ and they  generate the $*$-algebra $\C\cN_+$. Writing $z=x_1+{\ii }x_2$ with $x_1,x_2\in \R$, we get
 $$
 1+v(z)=1+\frac{x_1+{\ii}x_2}{x_1-{\ii}x_2} =2\,\frac{x_1^2+{\ii}\,x_1x_2}{x_1^2+x_2^2},~~ 1-v(z)= 2\, \frac{x_2^2-{\ii}\,x_1x_2}{x_1^2+x_2^2}\,.
 $$
This implies that the complex algebra $\C\cN_+$ is generated  by the  five functions
\begin{align}\label{fivef}
x_1,~~ x_2,~~ \frac{x_1^2}{x_1^2+x_2^2}\,,~~\frac{x_2^2}{x_1^2+x_2^2}\,,~~\frac{x_1x_2}{x_1^2+x_2^2}\,.
\end{align}
Obviously,  the hermitean part $(\C\cN_+)_h$ of the complex $*$-algebra\, $\C\cN_+$\, is just the {\it real} algebra generated by the functions  (\ref{fivef}).
This real algebra is the special case $d=2$ of the $*$-algebra $\cA$ treated in Section \ref{extensionsection}. Therefore, if we identify $\C$ with $\R^2$, the assertion of Theorem \ref{extcmpstochelsz} follows at once from  Theorem \ref{extensionthm}.

\section{Application to the two-sided complex moment problem}\label{twosdidecmp}

The two-sided complex moment problem is the moment problem for the $*$-semigroup $\Z^2$ with involution $(m,n):=(n,m)$.
Given a   sequence $s=(s_{m,n})_{(m,n)\in \Z^2}$ it asks when does there exist a positive Borel measure $\mu$ on $\C^\times:=\C\backslash \{0\}$  such that 
the function $z^m\ov{z}^n$ on $\C^\times$ is $\mu$-integrable and
\begin{align*}
s_{mn} =\int_{\C^\times}~ z^m\ov{z}^n\, d\mu(z)\quad {\rm for~all}\quad (m,n)\in \Z^2 .
\end{align*}
Note that this requires conditions for the measure $\mu$ at infinity and at zero. 

The following basic result was obtained by T.M. Bisgaard \cite{bisgaard}.
\begin{thm}\label{bisgaardthm}
A linear functional $L$ on $\C\Z^2$ is a moment functional if and only if $L$ is a positive functional, that is, $L(f^*f)\geq 0$ for all $f\in \C\Z^2$. 
\end{thm}
In terms of $*$-semigroups the main assertion of this theorem says that {\it each positive semidefinite sequence  on $\Z^2$ is a moment sequence on $\Z^2$.} This result  is somewhat surprising, since $C^\times$ has dimension $2$ and no additional condition (such as strong positivity or some appropriate extension) is required.

First we reformulate the semigroup $*$-algebra $\C\Z^2$.
Clearly, $\C\Z^2$ is generated by the functions $z, \ov{z}, z^{-1},\ov{z}^{\, -1}$ on the complex plane, that is, $\C\Z^2$ is the $*$-algebra $\C[z,\ov{z},z^{-1},\ov{z}^{-1}]$ of  complex Laurent polynomials in $z$ and $\ov{z}$. A vector space basis of $\cA_\C$ is the set $\{z^k\ov{z}^l; k,l\in \Z\}$. Writing $z=x_1+{\ii} x_2$ with $x_1,x_2\in \R$ we  have
$$
z^{-1}=\frac{x_1-{\ii} x_2}{x_1^2+x_2^2}~~~{\rm and}~~~{\ov{z}}^{\,-1}=\frac{x_1+{\ii} x_2}{x_1^2+x_2^2}.
$$
Hence $\C\Z^2$ is the complex  unital algebra generated by the four  functions 
\begin{align}\label{bisgaard}
x_1, ~~x_2,~~ y_1:= \frac{x_1}{x_1^2+x_2^2}\,,~~y_2:=\frac{x_2}{x_1^2+x_2^2}
\end{align}
on $\R^2\backslash \{0\}$. 
Note that
\begin{align}\label{yx2ide}
(y_1+{\ii}y_2)(x_1-{\ii }x_2)=1.
\end{align}
{\it Proof of Theorem \ref{bisgaardthm}:}\\
As above we identify $\C$ and $\R^2$ in the obvious way. As discussed at the end of Section \ref{propertiessmpandmpandfibretheorem}, the hermitean part of the complex $*$-algebra $\C \Z^2$ is a real algebra $\cA$.
First we determine the character set $\hat{\cA}$ of 
$\cA$. 
Obviously, the point evaluation at each point $x\in  \R^2\backslash \{0\}$ defines uniquely a character $\chi_x$ of $\cA$.  From (\ref{yx2ide}) it follows at once that there is no character $\chi$ on $\cA$ for which $\chi(x_1)=\chi(x_2)=0$. Thus,
$$
\hat{\cA}=\{\chi_x; x\in \R^2, x\neq 0\, \}.
$$
The three functions 
$$h_1(x)=x_1y_1=\frac{x_1^2}{x_1^2+x_2^2}\,,~  h_2(x)=x_2y_2=\frac{x_2^2}{x_1^2+x_2^2}\,,~ h_3(x)=2x_1y_2=\frac{x_1x_2}{x_1^2+x_2^2}$$ 
are elements of $\cA$ and bounded on $\hat{\cA}\cong\R^2\backslash \{0\}$. Therefore, arguing as  in the proof of Theorem \ref{mpshperesinrd} it follows that the preorders $\sum (\cA/\cI_\lambda)^2$ for all fibers satisfy (MP) and so does $\sum \cA^2$ by  Theorem \ref{fibreth1}. Since\, $\hat{\cA}\cong\R^2\backslash \{0\}=\C^\times$,  this gives the assertion. 
\hfill $\Box$
\smallskip

Remark. The  algebra $\cA$ generated by the four functions $x_1,x_2,y_,y_2$ on $\C^\times $ is an interesting structure: The generators satisfy   the  relations
$$
x_1y_1+x_2y_2=1\quad {\rm and}\quad (x_1^2+x_2^2)(y_1^2+y_2^2)=1 
$$
and there is a $*$-automorphism $\Phi$ of the real algebra $\cA$ (and hence of the complex $*$-algebra $\C\Z^2$) given by $\Phi(x_j)=y_j$ and $\Phi(y_j)=x_j$, $j=1,2$.

\medskip

\noindent{\bf Acknowledgement.} The author would like to thank  Tim Netzer for valuable discussions on the subject of this paper.

\bibliographystyle{amsalpha}

\end{document}